\documentclass[11pt]{amsart}

\usepackage{color}
\usepackage{changes}
\usepackage{enumerate}
\usepackage{ulem,ifthen,xcolor,xkeyval,pdfcolmk}

\definechangesauthor[color=red]{zl}

\usepackage{amssymb}

\usepackage{graphics}
\usepackage{graphicx}

\usepackage{amsmath}
\usepackage{amsthm}
\usepackage{amsfonts}
\usepackage{xcolor}
\usepackage[english]{babel}
\usepackage[margin=1.0in]{geometry}
\usepackage[colorlinks=true,
        linkcolor=blue]{hyperref}
\usepackage{bbm}
\usepackage{verbatim}
\usepackage[T1]{fontenc}

\newcommand{\E}{\mathcal{E}}
\newcommand{\F}{\mathcal{F}}
\newcommand{\R}{\mathbb{R}}
\newcommand{\B}{\mathcal{B}_a}
\newcommand{\LX}{{L^2(X;\ \mu)}}

\newcommand{\LXa}{{L^2(X_a)}}
\newcommand{\UT}{\tilde{U}}

\newcommand{\abs}[1]{\left|#1\right|}       
\newcommand{\N}[1]{||#1||}                  
\newcommand{\ang}[1]{\left<#1\right>}       
\newcommand{\brak}[1]{\left(#1\right)}      
\newcommand{\crl}[1]{\left\{#1\right\}}     
\newcommand{\edg}[1]{\left[#1\right]}       

\newtheorem{thm}{Theorem}[section]
\newtheorem{defi}{Definition}[section]
\newtheorem{definition}{Definition}[section]
\newtheorem{prop}{Proposition}[section]

\newtheorem{lem}{Lemma}[section]
\newtheorem{remark}{Remark}[section]

\newtheorem{assump}{Assumption}

\numberwithin{equation}{section}

\title{Powers of generators on Dirichlet Spaces and applications to Harnack principles}
\author{Fabrice Baudoin, Quanjun Lang, Yannick Sire}
\date{}

\begin{document}
\maketitle

\begin{abstract}
We provide a general framework for the realization of powers or functions of suitable operators on Dirichlet spaces. A first contribution is to unify the available results dealing with specific geometries; a second one is to provide new results on rather general metric measured spaces which were not considered before and falls naturally in the theory of Dirichlet spaces. The main tool is using the approach based on subordination and semi-groups by Stinga and Torrea. Assuming more on the Dirichlet space, we derive several applications to PDEs such Harnack and Boundary Harnack principles.  
\end{abstract}

\tableofcontents

\section{Introduction}

The last two decades  have seen an important amount of works whose aim is to realize powers of some operators $L$ in terms of a suitable extension. When the operator $L$ is second-order in divergence-form for instance, the extension appears to be a differential operator and classical tools from PDEs allow to get (or recover) several results on the operator $L^s$ (for $0<s<1$) such as regularity estimates and fine properties of solutions of an associated PDE. Functions of $L$ are of course multipliers in the sense of harmonic analysis and connections can be made with well-known results such as H\"ormander-Mikhlin theorems. On the other hand, powers of $L$ are a subclass of generators of Levy processes and some results which can be proved via probabilistic techniques can be recovered through PDE ones.

We now describe more precisely what we mean by extension. A classical result about the square root of the Laplace operator is the following:  if $u(x,y)$ is a harmonic function in the upper half-space  $\R^n\times \R^+$ with boundary value $f(x) = u(x,0)$, then under certain conditions on $u$, we have
$$ -\sqrt{-\Delta}f(x) = \frac{\partial}{\partial y} u(x,0).$$
 For $0<s<1$, the fractional Laplacian of a function $f:\R^n\rightarrow \R$ is defined via Fourier transform on the space of tempered distributions as 
$$
\widehat{(-\Delta)^s}f(\xi) = |\xi|^{2s}\widehat{f(\xi)}.
$$
As L. Caffarelli and L. Silvestre showed in \cite{Caffarelli_2007}, for a function $f:\R^n\rightarrow\R$, a solution $u$ to the following equation
\begin{align*}
    \begin{cases}
    \Delta_x u + \frac{1-2s}{y}u_y + u_{yy} = 0\\
    u(x,0) = f(x).
    \end{cases}
\end{align*}
can be given by a Poisson formula and one has the Dirichlet-to-Neumann condition 
$$
-(-\Delta)^s f(x) = C \lim_{ y \rightarrow 0^+} y^{1-2s}u_y(x,y)
$$
with the constant $C$ depends on the dimension $n$ and $s$. This allows to realize some powers of the Euclidean laplacian in terms of traces of a differential operator in the upper half-space. Written in divergence form the previous equation involves the weight $y^{1-2s}$ which belongs to the class $A_2$. As a consequence, one can apply the theory developed in \cite{2019arXiv191105619B, FKS,FJK,FJK2} to derive an Harnack inequality, a boundary Harnack principle and other results for equations of the type $(-\Delta)^s f(x)=0$.

In \cite{stingatorrea}, P. R. Stinga and J. L. Torrea develop a framework for a general non-negative self-adjoint operator $L$ on $L^2(\Omega, d\eta)$, with $\Omega$ being an open set in $\R^n$ and $d\eta$ a positive measure. It is remarkable that they applied the spectral theorem and semi-group theory instead of the Fourier transform, which requires much less structure on the ambient space than the Euclidean space. As a direct application of their framework, one can get the previously described realizations of $L^s$ on Riemannian and Sub-Riemannian manifolds under classical geometric assumptions such as polynomial volume for instance. 

M. Kwasnicki and J. Mucha 
\cite{Kwasnicki2018}
discussed the extension problem for complete Berstein functions of the Laplacian, other than just fractional powers. Their argument is based on the Fourier transform again and the Krein's spectral theory of strings. Krein's theory provides a one to one correspondence between the non-negative locally integrable measures (Krein's strings) and the complete Bernstein functions. 

One of the goals of the present paper is to consider the extension problem on rather general Dirichlet spaces. This allows to give a unified theory of the previously known results; but also to deal with new operators $L$ which were not considered before. Then invoking results in the literature for the extended PDE, one gets new results about regularity of solutions of some PDEs, in the same spirit as the ones described above. It is important here to notice that Dirichlet spaces are the natural metric spaces for which such extension theory holds, up to additional assumptions of course on the space if one wants to get additional results on the solutions of specific equations.  We refer the reader to \cite{DirichletFormsandSymmetricMarkovProcesses} for an extensive study of the theory of  Dirichlet spaces. In this general framework the Fourier transform is not always available and as a consequence we will adopt the strategy in \cite{stingatorrea} based on semi-groups. 

Due to the versatile nature of the semi-group approach of Stinga and Torrea, one needs very little assumptions on the Dirichlet space to realize the powers of the generator $L$ as the Dirichlet-to-Neumann operator of a suitable extension. On the contrary, getting fine properties  of the solutions (existence included actually) of the equation on the extension requires much more on the underlying Dirichlet space. 

The framework we adopt here covers the following geometries :

\begin{itemize}
\item Complete Riemannian manifolds with non-negative Ricci curvature or more generally RCD$(0,\infty)$ spaces in the sense of Ambrosio-Gigli-Savar\'e \cite{ambrosio2014}, 
\item Carnot groups and other complete sub-Riemannian manifolds satisfying a generalized curvature dimension inequality (see \cite{BAUDOIN20122646,BK}), 
\item Doubling metric measure spaces that support a $2$-Poincar\'e inequality with respect to the upper gradient
structure of Heinonen and Koskela (see~\cite{heinonen_koskela_shanmugalingam_tyson_2015,KOSKELA20142437,Koskela2012}).
\item Metric graphs with bounded geometry (see \cite{Haeseler}).
\item Abstract Wiener spaces are Dirichlet spaces (see \cite{bogachev}). 
\end{itemize}

The previous items concern mainly the extension property. In some cases, it was known to hold like on the Euclidean case with positive measure \cite{Caffarelli_2007,stingatorrea}, Heisenberg groups \cite{FF}, Riemannian manifolds with curvature assumptions \cite{stingatorrea}, abstract Weiner spaces and Gauss spaces \cite{NPS1,NPS2} and some variations of them like in bounded domains.

\begin{remark}
When this paper was finished, we have been aware that S. Eriksson-Bique, G. Giovannardi, R. Korte, N. Shanmugalingam and G. Speight obtained similar results in the context of metric measure spaces endowing with an upper gradient. In this more favorable setting than ours, they could get further regularity properties of the PDEs under consideration.  
\end{remark}

\section{Preliminaries on Dirichlet spaces}\label{Dirichlet spaces}


Here we provide an overall introduction to Dirichlet spaces. One can refer to the book of \cite{DirichletFormsandSymmetricMarkovProcesses}
for more details.
Let $(X,d)$ be a locally compact metric space equipped with a Radon measure $\mu$ supported on $X$. Let $(\E,\F = \mathcal{D}(\E))$ be a densely defined, symmetric bilinear form on $\LX$. Note that 
$$(u,v)_\F = (u,v)_{\LX} + \E(u,v)$$
is a  inner product on $\F$. 
Then we can define the norm on $\F$ by Cauchy-Schwartz inequality, 
$$\N{u}_{\F} = \brak{\E(u,u) + \N{u}_{\LX}^2}^{1/2}.$$
We say $\E$ is $closed$ if $\F$ is complete with respect to the norm $\N{\cdot}_\F$.
Given $\E$ is closed, we say it is $Markovian$ if 
$$
u \in \F, \,v \text{ is a normal contraction of  } u \Rightarrow v \in \F, \,\E(v,v) \leq \E(u,u).$$
Here a function $v$ is called  a  $normal\  contraction$ of a function $u$, if 
$$
\abs{v(x)-v(y)} \leq \abs{u(x)-u(y)},\ \forall x,y\in X, \ \abs{v(x)}\leq \abs{u(x)},\ \forall x \in X.
$$
\begin{defi}
We say $(\mathcal{E},\mathcal{F} = \mathcal{D}(\mathcal{E}))$ is a Dirichlet form on $L^2(X,\mu)$, if $\E$ is a densely defined, closed, symmetric and Markovian bilinear form on $L^2(X,\mu)$. 
\end{defi}

By the following theorem in \cite[Theorem 1.3.1]{DirichletFormsandSymmetricMarkovProcesses}, we can define the generator for a Dirichlet form.
\begin{defi}
There is a one to one correspondence between the family of closed symmetric forms $\E$ on a Hilbert space $H$and the family of non-positive definite self-adjoint operators $L$ on $H$. The correspondence is determined by
\begin{align*}
    \begin{cases}
    D\brak{\E} = D\brak{\sqrt{-L}}\\
    \E(u,v) = (\sqrt{-L}u,\sqrt{-L}v)
    \end{cases}
\end{align*}
$L$ is called the generator of the form $\E$.
\end{defi}

In the classical  Euclidean case \cite{Caffarelli_2007}, $L$ is the Laplacian operator, and $\E(u,v) = \int_{\R^n} \nabla u \cdot \nabla v \,dx$.

In the following sections, we will be focusing on the generators of the Dirichlet forms. Let us consider a non-positive definite self-adjoint operators $L$ on a Hilbert space $H$. In our case, $H$ will be $\LX$. By the spectral theorem, there exists a unique spectral family $dE(\lambda)$, such that 
$$ -L = \int_0^\infty \lambda dE(\lambda).$$
This formula is understood in the sense that, for any functions $f,g\in D\brak{L}$, we have 
$$\ang{-Lf, g} = \int_0^\infty \lambda dE_{f,g}(\lambda).$$
In particular, for any non-negative continuous function $\phi$ on $[0,\infty)$, we can define
\begin{align}
    \begin{cases}
    \phi(-L) = \int_0^\infty \phi(\lambda)dE(\lambda),\\
    D\brak{\phi(-L)} = \crl{u\in H : \int _0 ^\infty \phi(\lambda)^2 dE_{u,u}(\lambda)<\infty}. 
    \end{cases}
\end{align}

\section{Extension theorem  on Dirichlet spaces}\label{frac power}

Let $L$ be a non-positive symmetric operator defined on $D(L)$ generating the Dirichlet form $\mathcal E$. 
The heat semigroup associated to $L$ will be denoted by $P_t$. We assume that \underline{$L$ has no spectral gap}. 

\subsection{Fractional powers}

Here we consider $(-L)^s$, the fractional power of $L$, where $ 0 < s < 1$. It can be defined by the spectral theorem,
$$ (-L)^s = \int_0^\infty \lambda^s dE(\lambda).$$
Similarly to the result of \cite{stingatorrea}, we can extend the operator into a higher dimension space $X\times \R$, as it is shown in the following lemmas. We compute now powers of the generator $L$. 
The spectral theorem yields that for $f \in \mathcal{D}((-L)^s)$, with $s \in (0,1)$
\[
(-L)^s f  =\frac{1}{\Gamma(-s)} \int_0^{+\infty} (P_t f -f) \frac{dt}{t^{1+s}}
\]

If $f \in C(X) \cap \mathcal{D}(-L)$ and $Lf \in L^\infty(X,\mu)$, this expression can be interpreted pointwise everywhere. Indeed, one has 
\[
P_t f(x) -f(x)=\int_0^t L P_s f(x) ds=\int_0^t  P_s L f(x) ds
\]
so that 
\[
\left| P_t f(x) -f(x) \right| \le t \| Lf \|_\infty
\]
Thus, 
\begin{align*}
\int_0^{+\infty} | P_t f(x) -f(x)| \frac{dt}{t^{1+s}}&= \int_0^{1} | P_t f(x) -f(x)| \frac{dt}{t^{1+s}}+ \int_{1}^{+\infty} | P_t f(x) -f(x)| \frac{dt}{t^{1+s}} \\
 & \le C_1 \| Lf \|_\infty + C_2 \| f \|_2
\end{align*}

Note that we also have

\[
P_t f(x) -f(x)=\int_X p_t(x,y) (f(y)-f(x)) d\mu(y)
\]
So,
\[
(-L)^s f (x) =\int_X K(x,y) (f(y)-f(x)) d\mu(y)
\]
with
\[
K(x,y)=\int_0^{+\infty} p_t(x,y)  \frac{dt}{t^{1+s}}
\]
On Dirichlet spaces endowed with a  doubling measure  and a $2$-Poincar\'e inequality, the following Gaussian bounds for the heat kernel hold
\[
 p_t (x,y) \simeq   C\, \frac{e^{-c d(x,y)^2/t}}{\mu(B(x,\sqrt{t}))}
\]
Therefore, if we assume maximal volume growth, i.e. 
\[
\mu(B(x,r)) \ge C r^n
\]
we get
\[
K(x,y) \simeq \frac{1}{d(x,y)^{n+2s}}. 
\]

 Along with $L$ and, for any $-1<a<1$ we consider the Bessel operator
\begin{equation}\label{Ba}
\mathcal B_a = \frac{\partial^2 }{\partial y^2} + \frac{a}{y} \frac{\partial }{\partial y},
\end{equation}
on the whole line $\R$ endowed with the measure $d\nu_a(y) = |y|^a dy$.
Note that the correspondence between $w(y) = |y|^a$ and $\psi(\lambda) = \lambda ^ s$ is shown in the Krein's theory in Section \ref{psiL}.

Now let us consider the space $X_a = X \times \R$ with the measure $d\mu \otimes d\nu_a$. We will also denote by $X_a^+ = X \times (0,\infty)$, and by $X_a^- = X \times (-\infty,0)$.

\begin{lem}\label{PoiF}
    Let  $f \in \mathcal{D}((-L)^s)$. Then the $L^2$-weak solution of the extension equation
    \begin{equation}\label{extL}
    \begin{cases}
    L_a U = (L + \mathcal B_a) U = 0\ \ \ \ \ \ \ \text{in}\ X_a^+,
    \\
    U(\cdot ,0) = f ,
    \end{cases}
    \end{equation}
    where $a = 1- 2s$, is given by the following function
    \[
    U (\cdot ,y)=\frac{1}{\Gamma(s)} \int_0^{+\infty}   (P_t (-L)^s f) e^{-\frac{y^2}{4t}} \frac{dt}{t^{1-s}}
    \]
Moreover, we have the Poisson formula
\[
U (\cdot ,y)= \frac{y^{2s}}{2^{2s}\Gamma(s) }\int_0^{+\infty} (P_t f) e^{-\frac{y^2}{4t}} \frac{dt}{t^{1+s}}
    \]
Here the function $U$ is called the $s$-Harmonic extension of $f$.
\end{lem}

\begin{proof}
    1.\ We first show that for all $y>0$,  $U(\cdot, y) \in L^2(X;\mu)$, and for all $g \in L^2(X;\mu)$,

    $$ \left< U(\cdot \ , y), \ g(\cdot)\right>_{L^2(X;\ \mu)} = \frac{1}{\Gamma(s)}\int_0^\infty \left<P_t(-L)^sf, g\right>_{L^2(X;\mu)} e^{-\frac{y^2}{4t}}\frac{dt}{t^{1-s}}.$$

    For each $R>0$, we define 
    $$ U_R(x,y) = \frac{1}{\Gamma(s)}\int_0^R \left<P_t(-L)^sf, g\right>_{L^2(X;\mu)} e^{-\frac{y^2}{4t}}\frac{dt}{t^{1-s}}.$$
    Since $f \in \mathcal{D}((-L)^s)$, we have that $P_t(-L)^sf \in L^2(X;\ \mu)$, hence by Bochner theorem, $U_R$ is well-defined. Hence

    \begin{align*}
        \left< U_R(\cdot \ , y), \ g(\cdot)\right>_{L^2(X;\ \mu)} & = \frac{1}{\Gamma(s)}\int_0^R \left<P_t(-L)^sf, g\right>_{L^2(X;\mu)} e^{-\frac{y^2}{4t}}\frac{dt}{t^{1-s}}\\
        & = \frac{1}{\Gamma(s)}\int_0^R \int_0^\infty 
        e^{-t\lambda} \lambda^s dE_{f,g}(\lambda) 
        e^{-\frac{y^2}{4t}}\frac{dt}{t^{1-s}}\\
        & = \frac{1}{\Gamma(s)}\int_0^\infty \int_0^R e^{-t\lambda} (t\lambda)^s
        e^{-\frac{y^2}{4t}}\frac{dt}{t} dE_{f,g}(\lambda)\\
        & = \frac{1}{\Gamma(s)}\int_0^\infty \int_0^{R\lambda} e^{-r} r^s
        e^{-\frac{y^2 \lambda }{4r}}\frac{dr}{r} dE_{f,g}(\lambda).\\
    \end{align*}
    The change of integration follows from the integrability, and the last equality follows from the change of variable $r = t\lambda$. Hence we have
    \begin{align*}
        \left|\left<U_R(\cdot, y), \ g(\cdot)\right>_{L^2(X;\ \mu)} \right| &\leq
        \frac{1}{\Gamma(s)}\int_0^\infty \int_0^\infty e^{-r} r^s
        \frac{dr}{r} d\left|E_{f,g}(\lambda)\right|\\
        & = \frac{1}{\Gamma(s)} \int_0^\infty e^{-r} r^s
        \frac{dr}{r} \int_0^\infty d\left|E_{f,g}(\lambda)\right|\\
        & \leq \N{f}_{L^2(X; \mu)} \N{g}_{L^2(X; \mu)}.
    \end{align*}
    Therefore, for each fixed $y >0$, $U_R(\cdot, y)$ is in $L^2(X; \mu)$ and 
    $$ \N{U_R(\cdot, y)}_{L^2(X; \mu)} \leq \N{f}_{L^2(X; \mu)}.$$
    
    And by the similar computation, for some $R_2 > R_1 > 0$, 
    \begin{align*}
        \abs{\ang{U_{R_1}(\cdot, y), g} - \ang{U_{R_2}(\cdot, y), g}}
        & \leq 
        \frac{1}{\Gamma(s)} \int_0^\infty e^{-r} r^s
        \frac{dr}{r} \int_{R_1}^{R_2} d\left|E_{f,g}(\lambda)\right| \to \ 0
    \end{align*}
    as $R_1, \ R_2 \to 0$. Hence there exist a Cauchy sequence of bounded operators $\crl{U_{R^j}(\cdot, \ y)}_{j \in \mathbb{N}}$ in $L^2(X; \ \mu)$, and it converge to $U(\cdot, \ f)$ which is defined earlier in  weakly in $L^2(X;\ \mu)$ as $R^j \to \infty$.  Moreover, by the standard dominated convergence theorem, 
    \begin{align*}
        \left< U(\cdot \ , y), \ g(\cdot)\right>_{L^2(X;\ \mu)} &= 
        \lim_{R^j \to \infty} \left< U_{R^j}(\cdot \ , y), \ g(\cdot)\right>_{L^2(X;\ \mu)}\\
        &=
        \lim_{R^j \to \infty}
        \frac{1}{\Gamma(s)}\int_0^\infty \int_0^{R^j} e^{-t\lambda} (t\lambda)^s
        e^{-\frac{y^2}{4t}}\frac{dt}{t} dE_{f,g}(\lambda)\\
        &=
        \frac{1}{\Gamma(s)}\int_0^\infty \int_0^\infty e^{-t\lambda} (t\lambda)^s
        e^{-\frac{y^2}{4t}}\frac{dt}{t} dE_{f,g}(\lambda)\\
        &=
        \frac{1}{\Gamma(s)}\int_0^\infty \int_0^\infty e^{-t\lambda} \lambda^s dE_{f,g}(\lambda)
        e^{-\frac{y^2}{4t}}\frac{dt}{t^{1-s}}\\
        &=
        \frac{1}{\Gamma(s)}\int_0^\infty 
        \left<P_t(-L)^sf, g\right>_{L^2(X;\mu)}
        e^{-\frac{y^2}{4t}}\frac{dt}{t^{1-s}}.
    \end{align*}
    Hence we get the desired formula.
    
    2.\ Next we show that $U(\cdot, y)\in \text{Dom}(L)$, that is, 
    $$
    \lim_{r\to 0^+} \ang{
    \frac{e^{rL} U(\cdot, \ y) - U(\cdot, \ y)}{r}, \ g
    }_{L^2(X;\ \mu)}
    $$
    exists for all $g\in \LX.$
    
    Since $P_r = e^{rL}$ is self adjoint, 
    \begin{align*}
        \ang{P_r U(\cdot, \ y),\ g}_{\LX} &= \ang{ U(\cdot, \ y),\ P_r g}_{\LX}\\
        &=
        \frac{1}{\Gamma(s)}\int_0^\infty 
        \left<e^{tL}(-L)^sf, e^{rL}g\right>_{L^2(X;\mu)}
        e^{-\frac{y^2}{4t}}\frac{dt}{t^{1-s}}\\
        &=
        \frac{1}{\Gamma(s)}\int_0^\infty 
        \left<e^{(t+r)L}(-L)^sf, g\right>_{L^2(X;\mu)}
        e^{-\frac{y^2}{4t}}\frac{dt}{t^{1-s}}.\\
    \end{align*}
    That implies
    \begin{align*}
        \ang{\frac{e^{rL} U(\cdot, \ y) - U(\cdot, \ y)}{s}, \ g
        }_{L^2(X;\ \mu)} &= 
        \frac{1}{\Gamma(s)}\int_0^\infty 
        \ang{
        \frac{e^{(r+t)L}(-L)^s f-e^{tL}(-L)^s f}{r}, g
        }_{L^2(X;\mu)}
        e^{-\frac{y^2}{4t}}\frac{dt}{t^{1-s}}\\
        &=
        \frac{1}{\Gamma(s)}\int_0^\infty 
        \int_0^\infty
        \frac{e^{-(r+t)\lambda}\lambda^s-e^{-t\lambda}\lambda^s}{r}
        dE_{f,g}(\lambda)
        e^{-\frac{y^2}{4t}}\frac{dt}{t^{1-s}}\\
        &=
        \frac{1}{\Gamma(s)}\int_0^\infty 
        \int_0^\infty
        \frac{e^{-(r+t)\lambda}\lambda^s-e^{-t\lambda}\lambda^s}{r}
        e^{-\frac{y^2}{4t}}\frac{dt}{t^{1-s}}
        dE_{f,g}(\lambda)\\
        &\underrightarrow{r \to 0^+}\ 
        \frac{1}{\Gamma(s)}\int_0^\infty 
        \int_0^\infty
        \partial_r(e^{t\lambda})\lambda^s
        e^{-\frac{y^2}{4t}}\frac{dt}{t^{1-s}}
        dE_{f,g}(\lambda)\\
        &=
        \frac{1}{\Gamma(s)}\int_0^\infty 
        \ang{
        Le^{tL}(-L)^s f,\ g
        }_{L^2(X;\ \mu)}
        e^{-\frac{y^2}{4t}}\frac{dt}{t^{1-s}}.\\
    \end{align*}
    
    3. The boundary condition holds. By using the result from step 1 and change of variables, we can have that for all 
    $g \in L^2(X;\ \mu)$,
    \begin{align*}
        \ang{U(\cdot, y), g(\cdot)}_\LX = &
        \frac{1}{\Gamma(s)}\int_0^\infty \int_0^\infty 
        e^{-r}r^s
        e^{-\frac{y^2 \lambda}{4r}}\frac{dr}{r}
        dE_{f,g}(\lambda)\\
        \underrightarrow{y\to 0+}&\ 
        \frac{1}{\Gamma(s)}\int_0^\infty \int_0^\infty 
        e^{-r}r^s
        e^{-\frac{y^2 \lambda}{4r}}\frac{dr}{r}
        dE_{f,g}(\lambda)\\
        =&
        \ang{f,\ g}_\LX
    \end{align*}
    
    4. Now we are left to show that $U$ satisfied the equation (\ref{extL}). 
    For all $g \in L^2(X;\ \mu)$,
    \begin{align*}
        \lim_{h\to 0^+}\ang{
        \frac{U(\cdot, \ y + h) - U(\cdot, \ y)}{h},\ g(\cdot)
        }_\LX
        &=
        \frac{1}{\Gamma(s)}\int_0^\infty 
        \ang{e^{tL}(-L)^s f, \ g}_\LX 
        \partial_y(e^{-\frac{y^2}{4t}})
        \frac{dt}{t^{1-s}}\\
        &=
        \ang{
        \frac{1}{\Gamma(s)}\int_0^\infty 
        e^{tL}(-L)^s f \cdot
        \partial_y(e^{-\frac{y^2}{4t}})
        \frac{dt}{t^{1-s}},\ g
        }_\LX
    \end{align*}
    The first equality follows from the dominated convergence theorem, and the second holds by checking the integrability as in step 1. Hence
    $$ U_y(x, y) = \frac{-1}{\Gamma(s)}\int_0^\infty 
        e^{tL}(-L)^s f(x) \cdot
        \frac{
        ye^{-\frac{y^2}{4t}}
        }{2t}
        \frac{dt}{t^{1-s}}.$$
    Also, we can have $U_{yy}$ by similar computation,
    $$ U_{yy}(x, y) = \frac{1}{\Gamma(s)}\int_0^\infty 
        e^{tL}(-L)^s f(x) \cdot
        \brak{
        \frac{y^2}{4t^2} - \frac{1}{2t}
        }
        e^{-\frac{y^2}{4t}}
        \frac{dt}{t^{1-s}}.$$
    Hence for all $g \in L^2(X;\ \mu)$,
    \begin{align*}
        \ang{\mathcal B_a U,\ g} &= \ang{U_{yy} + \frac{1-2y}{y} U_y,\ g}\\
        &=
        \frac{1}{\Gamma(s)}\int_0^\infty 
        \ang{e^{tL}(-L)^s f, \ g}_\LX 
        \brak{
        \frac{y^2}{4t^2}-\frac{1}{2t}+\frac{1-2s}{y}\brak{-\frac{y}{2t}}}
        e^{-\frac{y^2}{4t}}
        \frac{dt}{t^{1-s}}\\
        &=
        \frac{1}{\Gamma(s)}\int_0^\infty 
        \ang{P_t(-L)^s f, \ g}_\LX 
        \brak{
        \frac{y^2}{4t^2}+\frac{s-1}{2t}
        }
        e^{-\frac{y^2}{4t}}
        \frac{1}{t^{1-s}}dt\\
        &=
        \frac{1}{\Gamma(s)}\int_0^\infty 
        \ang{P_t(-L)^s f, \ g}_\LX 
        \partial_t\brak{
        e^{-\frac{y^2}{4t}}
        \frac{1}{t^{1-s}}}dt
    \end{align*}
    And an integration by parts yields that,
    \begin{align*}
        \ang{\mathcal B_a U,\ g} 
        &=
        -\frac{1}{\Gamma(s)}\int_0^\infty 
        \partial_t
        \edg{
        \int_0^\infty
        e^{-t\lambda} \lambda^s dE_{f,g}(\lambda)
        }
        e^{-\frac{y^2}{4t}}
        \frac{dt}{t^{1-s}}\\
        &=
        \frac{1}{\Gamma(s)}\int_0^\infty 
        \int_0^\infty
         \lambda
        e^{-t\lambda} \lambda^s 
        e^{-\frac{y^2}{4t}}
        \frac{dt}{t^{1-s}}
        dE_{f,g}(\lambda)
        \\
        &=
        \ang{
        L\ \frac{1}{\Gamma(s)}\int_0^\infty
        e^{-tL}(L^sf)
        e^{-\frac{y^2}{4t}}
        \frac{dt}{t^{1-s}}, g}_\LX = \ang{LU(\cdot, y),g(\cdot)}_\LX.
    \end{align*}
    
\item 5. We are left to show the Poisson Formula. Again by change of variable $t = \frac{y^2}{4r\lambda}$, we get
\begin{align*}
    \ang{
    U(\cdot, y), g(\cdot)
    }_\LX &= 
    \frac{1}{\Gamma(s)}\int_0^\infty \int_0^\infty 
    e^{-t\lambda} (t\lambda) ^s 
    e^{-\frac{y^2}{4t}}
    \,\frac{dt}{t}\,dE_{f,g}(\lambda)\\
    &= 
    \frac{1}{\Gamma(s)}\int_0^\infty \int_0^\infty 
     e^{-\frac{y^2}{4r}} \brak{\frac{y^2}{4r}} ^s e^{-r\lambda}
    \,\frac{dr}{r}\,dE_{f,g}(\lambda)\\
    &=
    \frac{y^{2s}}{4^s \Gamma(s)}\int_0^\infty
    \ang{e^{-tL}f,g}_\LX
    e^{-\frac{y^2}{4r}}  
    \,\frac{dr}{r^{1+s}}\\
    &=
    \ang{\frac{y^{2s}}
    {4^s \Gamma(s)}\int_0^\infty
    e^{-tL}f
    e^{-\frac{y^2}{4r}}  
    \,\frac{dr}{r^{1+s}}
    ,g}_\LX
\end{align*}
The last equality follows from the Bochner's Theorem.
\end{proof}
Here $U$ is a solution to the equation (\ref{extL}) in $X_a$ with Dirichlet initial condition $U(\cdot, 0) = f(\cdot).$ The value of $(-L)^sf$  can be transformed to a Neumann initial condition of $U$.

\begin{lem}\label{dn}
 Let  $f \in \mathcal{D}((-L)^s)$. One can recover $(-L)^sf  $ by the following weighted Dirichlet-to-Neumann relation: 
\begin{equation}
(-L)^s f = 
- \frac{2^{2s-1} \Gamma(s)}{\Gamma(1-s)} \underset{y\to 0^+}{\lim} y^a \frac{\partial U}{\partial y}(\cdot,y) ,
\end{equation}
where, as above, $a = 1-2s$, and the identity holds in $L^2$.
\end{lem}
\begin{proof}
By the previous computation, for all $g \in L^2(X;\mu)$
\begin{align*}
    \ang{
    y^a U_y(\cdot, y), g(\cdot)
    }_\LX &= 
    \frac{1}{\Gamma(s)}\int_0^\infty 
    \ang{e^{tL}(-L)^s f, \ g}_\LX 
    y^a\frac{y}{2t}
    e^{-\frac{y^2}{4t}}
    \frac{dt}{t^{1-s}}
\end{align*}
Change the variable $t = \frac{y^2}{4r}$,
\begin{align*}
    \ang{
    y^a U_y(\cdot, y), g(\cdot)
    }_\LX =& 
    \frac{-1}{\Gamma(s)}\int_0^\infty \int_0^\infty 
    e^{-t\lambda}\lambda^s 
    \frac{y^{2-2s}}{2t}
    e^{-\frac{y^2}{4t}}
    dE_{f,g}(\lambda)
    \frac{dt}{t^{1-s}}\\
    =&
    \frac{-1}{\Gamma(s)}\int_0^\infty \int_0^\infty 
    e^{-\frac{y^2\lambda}{4r}}\lambda^s 
    dE_{f,g}(\lambda)
    \frac{2e^{-r}}{(4r)^s}dr\\
    \underrightarrow{y \to 0^+}&
    \frac{-1}{\Gamma(s)}\int_0^\infty \int_0^\infty
    \lambda^s 
    dE_{f,g}(\lambda)
    \frac{2e^{-r}}{(4r)^s}dr\\
    =&
    \frac{-1}{\Gamma(s)}2^{1-2s} \int_0^\infty 
    r^{-s}e^{-r}dr \cdot \ang{(-L)^sf, g}_\LX\\
    =&
    -\frac{\Gamma(s)2^{2s-1}}{\Gamma(1-s)}\ang{(-L)^sf, g}_\LX
\end{align*}
\end{proof}
\subsection{General weights and Functions of L}\label{psiL}

In this section, we generalize the results in the previous section to some functions of the generator $L$, which are not necessarily powers of it. We first state the following theorem (see \cite{Kwasnicki2018}). 

\begin{thm}
 Let $A$ be a Krein's string, i.e. non-negative locally integrable function on $[0,{+\infty})$. Then for every $\lambda \geq 0$, there exists a unique non-increasing function $R(z,\lambda)$ on $[0,{+\infty})$ which solves
 \begin{equation}
     \begin{cases}
        R_{zz}(z,\lambda) = \lambda A(z) R(z,\lambda),\ \text{for  } z>0\\
        R(0,\lambda) = 1,\ \text{for all } \lambda>0\\
        \lim_{z\rightarrow {+\infty}}R(z,\lambda) \geq 0.
        \end{cases}
 \end{equation}
 (with the second derivative understood in the weak sense). Furthermore, the expression
 $$
 \psi(\lambda) = -R_z(0,\lambda).$$
 defines a complete Bernstein function $\psi$, and the correspondence between $A(s)$ and $\psi(\lambda)$ is one-to-one.
\end{thm}

We now consider the following extension problem on $X\times\R^+$

\begin{equation} \label{extALv}
    \begin{cases}
    A(z)Lv(x,z) + v_{zz}(x,z) = 0, \ \text{ in } X\times \R^+,\\
    v(x,0) = f(x),\ \text{ in  } X.
    \end{cases}
\end{equation}
By the change of variable $z = \sigma(y)$, where $\sigma(y) = \int_0^y \frac{1}{w(r)}dr$, we have $A(z) = A(\sigma(y)) = (w(y))^2$. Then we can recover the equation
\begin{equation}
    \begin{cases}
    Lu(x,y) +\frac{1}{w(y)}\frac{\partial}{\partial y}\brak{w(y)\frac{\partial}{\partial y} u(x, y)} = 
    Lu(x,y) +\frac{w'(y)}{w(y)}u_y(x, y) + u_{yy}(x,y) =
    0,\ \text{ in }\ X\times \R^+,
    \\
    u(x ,0) = f(x) \ \text{ in } X.
    \end{cases}
\end{equation}
Just like the previous case, we have the following Poisson formula for the equation. 
\begin{equation}\label{vrf}
    v(x,z) = R(z,-L)f(x)
\end{equation}

\begin{thm}
For $f\in D(\psi(-L))$, the formula \eqref{vrf} is a $L^2$-weak solution of the equation \eqref{extALv}. In particular, for all $g \in L^2(X,\mu)$, the following equation holds
\begin{equation}\label{weaksol}
    \ang{A(z)Lv(x,z) + v_{zz}(x,z),g(x)} = 0.
\end{equation}
Moreover, the Dirichlet to Neumann condition holds weakly.
\begin{equation}
    \psi(-L)f(x) = -\lim_{z\rightarrow 0}v_z(x,z).
\end{equation}

\end{thm}

\begin{proof}
By the spectral theorem, we can write
\begin{align*}
    \ang{A(z)L v(x,z)+v_{zz}(x,z),g(x)} &= \ang{A(z) L R(z,-L)f(x) + R_{zz}(z,-L)f(x) ,g(x)}\\
    &=\int_0^{+\infty}  \edg{A(z) \lambda R(z,\lambda) + R_{zz}(z,\lambda)} dE_{f,g}(\lambda)\\
    &=0.
\end{align*}

\begin{align*}
\ang{v(x,0),g(x)} &= \ang{R(0,-L)f(x),g(x)}= \int_0^{+\infty} R(0,\lambda) dE_{f,g}(\lambda)= \int_0^{+\infty} 1 dE_{f,g}(\lambda)= \ang{f(x),g(x)}.
\end{align*}

\begin{align*}
    \ang{v_z(x,0),g(x)} &= \ang{R_z(0,-L)f(x),g(x)}= \int_0^{+\infty} R_z(0,\lambda) dE_{f,g}(\lambda)= \int_0^{+\infty} -\psi(\lambda) dE_{f,g}(\lambda)= -\ang{\psi(-L)f(x),g(x)}.
\end{align*}

\end{proof}

Given the existence and uniqueness of $R$, one can define $G$ to be the $\psi$ times the inverse Laplace transform of $R$, i.e.
$$
R(z,\lambda) = \int_0^{+\infty} e^{-t\lambda} G(z,t) \psi(\lambda) dt.
$$
Then we can write
$$
v(x,z) = \int_0^{+\infty} e^{tL} \psi(-L)f(x) G(z,t) dt.
$$
where $G(z,t)$ is precisely the heat kernel to the equation
$$
G_{zz}(z,t) = A(z)G_t(z,t).$$

\section{Applications to Harnack inequalities}

In this section, we will prove the Harnack inequality for the solution of the equation $$
(-L)^s f = 0.$$

%
%

\subsection{Harmonic functions on  Dirichlet spaces}

We denote by $C_c(X)$ the space of all continuous functions with compact support in $X$ and $C_0(X)$ its closure with respect to the supremum norm. A $core$ of $\E$ is a subset $\mathcal{C}$ of $\F \cap C_0(X)$ such that $\mathcal{C}$ is dense in $\F$ with the norm $\N{\cdot}_\F$ and dense in $C_0(X)$ with the supremum norm.

\begin{defi}
A Dirichlet form $\E$ is called $regular$ if it admits a core. 
\end{defi}

\begin{defi}
A Dirichlet form $\E$ is called strongly  local if for any $u,v\in\F$ with compact support, $v$ is constant on a neighbourhood of the support of $u$, then $\E(u,v) = 0$.
\end{defi}

Throughout this section, we assume that $(\mathcal{E},\mathcal{F})$ is a strongly local regular Dirichlet form on $\LX$.  Since  $\mathcal{E}$ is regular, the following definition is valid.

\begin{defi}
Suppose $\E$ is a regular Dirichlet form, for every $u,v\in \mathcal F\cap L^{\infty}(X)$, the energy measure $\Gamma (u,v)$ is defined  through the formula
 \[
\int_X\phi\, d\Gamma(u,v)=\frac{1}{2}[\mathcal{E}(\phi u,v)+\mathcal{E}(\phi v,u)-\mathcal{E}(\phi, uv)], \quad \phi\in \mathcal F \cap C_c(X).
\]
\end{defi}
Note that $\Gamma(u,v)$ can be extended to all $u,v\in \mathcal F$ by truncation (see \cite[Theorem 4.3.11]{SymmetricMarkovProcessesTimeChangeandBoundaryTheory}). According to Beurling and Deny~\cite{beurling1958}, one has then   for $u,v\in \mathcal{F}$
\[
\mathcal E(u,v)=\int_X d\Gamma(u,v)
\]
and $\Gamma(u,v)$ is a signed Radon measure often called the energy measure. 

If $U \subset X$ is an open set, we define
\[
\mathcal{F}_{loc}(U)=\left\{ f \in L^2_{loc}(U), \text{ for every relatively compact } V \subset U , \, \exists f^* \in \mathcal{F}, \, f^*_{\mid V}=f_{\mid V}, \, \mu \, a.e. \right\}
\]
For $f,g \in \mathcal{F}_{loc}(U)$, on can define $\Gamma (f,g)$ locally by $\Gamma (f,g)_{\mid V}=\Gamma( f^*_{\mid V},g^*_{\mid V}) $.

\begin{defi}
Let $U \subset X$ be an open set. A function $f \in \mathcal{F}_{loc}(U)$ is called harmonic in $U$ if for every function $h \in \mathcal{F}$ whose essential support is included in $U$, one has
\[
\mathcal{E} (f, h)=0.
\]
\end{defi}

\subsection{Elliptic Harnack inequality for harmonic functions}

In this section we recall some known results about Harnack inequalities for harmonic functions. The main assumption is the volume doubling property and the existence of nice heat kernel estimates.

\begin{defi}\label{VD}
We say that the metric measure space $(X,d,\mu)$ satisfies the volume doubling property if 
there exists a constant $C>0$ such that for every $x\in X$ and $r>0$,
\[
\mu(B(x,2r))\le C\, \mu(B(x,r)).
\]
\end{defi}
%
%
%

\begin{defi}\label{PI}
We say that $(X,\mathcal{E})$ satisfies the 2-Poincar\'e inequality if there exist constants $C$, $\lambda >1$, such that for any ball $B$ in $X$ and $u\in \mathcal{F}$, we have 
\begin{equation}
\frac{1}{\mu(B)}\int_B \abs{u-u_B} d\mu \leq C \,\text{rad}\,(B) \brak{\frac{1}{\mu(\lambda B)}\int_{\lambda B} d\Gamma(u,u)}^{1/2}
\end{equation}

\end{defi}
We have then the following well-known result (see \cite{heinonen_koskela_shanmugalingam_tyson_2015}).

\begin{thm}\label{Harnack elliptic}
Assume that $(X,d,\mu,\mathcal{E})$ satisfies the doubling condition and the 2-Poincar\'e inequality. There exist a constant $C>0$ and $\delta \in (0,1)$ such that for any ball $B(x,R)\subset X$ and any non-negative function $u \in \mathcal{F}$ which is harmonic on  $B(x,R)$
\[
 \sup_{z \in B(x,\delta R)} u(z) \le C \inf_{z \in B(x,\delta R)} u(z),
\]
where by $\sup$ and $\inf$ we mean the essential supremum and essential infimum.
\end{thm}

\subsection{Properties of the extended Dirichlet space}

Associated to $\mathcal E$ we can define the bilinear form $\E_a$ on the space $X_a = X \times \R$ with domain $\F_a$,
$$\E_a(u,v) = \int_\R\E(u,v)d\nu_a + \int_{X_a} u_y\cdot v_y \ d \nu_a d\mu,$$
$$\F_a = \crl{
u \in L^2(X_a, d\mu \times d\nu_a),\  \E_a(u,u) < \infty
}.$$
As before, we define the norm
$$
\N{u}_{\E_a}^2 = \N{u}_\LXa ^2 + \E_a(u,u).
$$
\begin{prop}\label{prop:Ea}
$(X_a, d\mu\times d\nu_a, \E_a, \F_a)$ is a strongly local and regular Dirichlet Space, where $-1<a<1$. 
\end{prop}
\begin{proof}
One can easily derive the Markovian property and the strong local property. And the density of $\F_a$ follows from the regularity. We are left to show that 
\begin{enumerate}
    \item $\E_a$ is closed, which is equivalent to say $(\F_a, \N{\cdot}_{\E_a})$ is a Banach Space. 
    \item $\E_a$ is regular.
\end{enumerate}

1.\ For the closedness, given a Cauchy sequence $\crl{u_n}$ in $\F_a$, we want to show that there exists $u\in \F_a$ such that $u_n\rightarrow u$ in $\F_a$. Notice that 
    \begin{align*}
        \N{u_n-u_m}_{\F_a}^2 &= \int_\R \int_X |u_n-u_m|^2 d\mu d\nu_a + \int_\R \E(u_n-u_m,u_n-u_m) d\nu_a + \int_X\int_\R |\partial_y u_n-\partial_yu_m|^2 d\nu_a d\mu\\
        &= \int_\R \N{u_n-u_m}_\E^2 d\nu_a +\int_X\int_\R |\partial_y u_n-\partial_yu_m|^2 d\nu_a d\mu.
    \end{align*}
    Since $\F$ is a Banach space, $u_n(\cdot, y)$ can be viewed as a function maps from $\R$ to $\F_a$. In particular, $\crl{u_n(\cdot, y)}$ is a Cauchy sequence in $L^2(\R,\F;\ \nu_a)$. There exists $u$ in $L^2(\R,\F;\ \nu_a)$, such that $u_n\rightarrow u $. Notice that we also have $u_n\rightarrow u $ in $\LXa$.
    And $\crl{\partial_y u_n}$ being a Cauchy sequence in $\LXa$ implies there exists $u^y \in \LXa$ such that $\partial_y u_n \rightarrow u^y$. Now we are left to show that $\partial_y u = u^y$. We recall that the weak derivative for a Bochner  integrable function $h \in L^2(\R, L^2(X);\nu_a)$ is $\partial_y h$, if for any $\phi \in C_c^\infty(\R)$, 
    $$
    \int_\R h \phi' d\nu_a = -\int_\R \partial_y h \phi d\nu_a
    $$
    The equality holds in the sense of $L^2(X)$. Then for any $\phi(y) \in C_c^\infty(\R)$ and  $\xi(x) \in C_c(X)$, notice that $\xi(x)\phi(y) \in \LXa$,
    \begin{align*}
        \int_X \int _\R u(x,y)\xi(x)\phi'(y) d\nu_a d\mu
        &=
        \lim_{n\rightarrow \infty}\int_X \int _\R u_n(x,y)\xi(x)\phi'(y) d\nu_a d\mu\\
        &=
        \lim_{n\rightarrow \infty}\int_X \int _\R \partial _y u_n(x,y)\xi(x)\phi(y) d\nu_a d\mu\\
        &=
        \int_X \int _\R u^y(x,y)\xi(x)\phi(y) d\nu_a d\mu\\
    \end{align*}
    The limits result from $u_n\rightarrow u$ and $\partial_y u_n \rightarrow u^y$ in $\LXa$. Since $u(x,y),u^y(x,y) \in \LXa$, we have 
    $$\int_\R u(x,y)\phi'(y) d\nu_a = \int _\R u^y(x,y) \phi(y) d\nu_a\  a.e. \text{ in } X.$$
    Then by definition, $\partial_y u(x,y) = u^y(x,y)$.

    2.\ To show $\E_a$ is regular, we claim that $\mathcal{C}_a = \mathcal{C} \otimes H^1(\R) \subset \F_a$ is a core of $\E_a$, where $\mathcal{C}$ is the core of $\E$ and $H^1(\R)$ is the Sobolev space over $\R$. Given a function $f(x,y) \in \F_a$, suppose for any $\varphi(x) \in \mathcal{C}$ and $\psi(y) \in H^1(\R)$, 
    \begin{align*}
        (f(x,y), \varphi(x)\psi(y))_{\F_a}=0
    \end{align*}
    Notice that 
    \begin{align*}
         (f(x,y),\varphi(x)\psi(y))_{\F_a}&= \int_{X_a} f(x,y)\varphi(x)\psi(y) d\nu_a d\mu + 
    \int_{\R} \E(f(x,y)\varphi(x))\psi(y) d\nu_a \\
    &\ \ \ +
    \int_{X_a} \varphi(x)\partial_y f(x,y) \psi'(y) d\nu_ad\mu\\
    &= \ang{\ang{f(x,y), \varphi(x)}_\F, \psi(y)}_{H^1(\R)}.
    \end{align*}
    Then by the density of $\mathcal{C}$, we have $f \equiv 0$ a.e. with respect to $\mu\otimes \nu_a$ on $X_a$. Hence $\mathcal{C}_a$ is a core of $\E_a$.

\end{proof}

\begin{thm}\label{ext:VD}
Suppose $(X,d\mu)$ has the volume doubling property, so does $(X_a, \mu_a)$. Here
$X_a = X\times \mathbb{R}$ and $ d\mu_a = d\mu\times d\nu_a$, where $d\nu_a(y) = \abs{y}^ady.$
\end{thm}

\begin{proof}
The proof for the doubling property for  $(\R, d\nu_a)$ follows from \cite{FKS} for instance. 
Let $B_a$ be a ball in $X_a$ centered at $(x_0, y_0)$ of radius $R$, $2B_a$ is a co-centered ball of radius $2R$. Denote the projection of $B_a$ onto $X$ and $\R$ to be $B$ and $I$ respectively, and $D = R\times I$. It is clear that $2B_a \subset 2D$. And one can find small enough $\lambda < 1$ such that $\lambda D \subset B_a$. Then
\begin{align*}
    \mu_a(2B_a) \leq \mu_a(2D) = \mu(2B)\nu_a(2I)\leq C \mu(\lambda B)\nu_a(\lambda I) = C\mu_a(\lambda D)\leq \mu_a(B_a).
\end{align*}
\end{proof}

\begin{thm}\label{ext:PI}
Suppose the space $(X,\mu)$ has the Poincar\'{e}'s inequality, i.e. there exist constants $C$, $\lambda >1$, such that for any ball $B$ in $X$ and $u\in \mathcal{F}$, we have 
\begin{equation}
\frac{1}{\mu(B)}\int_B \abs{u-u_B} d\mu \leq C \,\text{rad}\,(B) \brak{\frac{1}{\mu(\lambda B)}\int_{\lambda B} d\Gamma(u,u)}^{1/2}
\end{equation}
so does the space $(X_a, \mu_a)$, in particular, there exist constants $\tilde{C}$, $\tilde{\lambda} >1$, such that for any ball $B_a$ in $X_a$ and $u\in \mathcal{F}_a$, we have 
\begin{equation}\label{poincare}
\frac{1}{\mu(B_a)}\int_{B_a} \abs{u-u_{B_a}} d\mu_a \leq \tilde{C} \,\text{rad}\,(B_a) \brak{\frac{1}{\mu(\tilde{\lambda} B_a)}\int_{\tilde{\lambda} B_a} d\Gamma_a(u,u)}^{1/2}
\end{equation}
\end{thm}

\begin{proof}
First we claim that the following statement is equivalent to \eqref{poincare}:

There exist constants $C$, $\lambda >1$, and for any ball $B_a$ in $X_a$ and $u\in \mathcal{F}_a$, there is a constant $c = c(u,B_a)$, such that we have 
\begin{equation}\label{poincarec}
\frac{1}{\mu(B_a)}\int_{B_a} \abs{u-c} d\mu_a \leq C \,\text{rad}\,(B_a) \brak{\frac{1}{\mu(\lambda B_a)}\int_{\lambda B_a} d\Gamma_a(u,u)}^{1/2}
\end{equation}
It's clear that \eqref{poincare}$\implies$\eqref{poincarec}. To show the opposite, suppose \eqref{poincarec} is correct. Notice that
\begin{align*}
    \frac{1}{\mu_a(B_a)}\int_{B_a} \abs{c-u_{B_a}} d\mu_a 
    &= \abs{c-\frac{1}{\mu_a(B_a)}\int_{B_a} u \,d\mu_a}\\
    &\leq \frac{1}{\mu_(B_a)}\int_{B_a} \abs{u-c} d\mu_a,
\end{align*}
Then \eqref{poincare} follows from triangle inequality and \eqref{poincarec} it self. 

We now have the freedom to choose the constant $c$. Let $B = B_a \cap X$, $I = B_a \cap \mathbb{R}$ and $D = B\times I$. Denote $$u_B(y) = \frac{1}{\mu(B)}\int_B u(x,y) d\mu(x),$$and $u_D = \frac{1}{\mu_a(D)}\int_D u d\mu_a = (u_B)_I$. Let $c = u_D$, and consider the region $D$ for now, which is to be fixed later. 
\begin{align*}
\frac{1}{\mu_a(D)}\int_{D} \abs{u-u_D} d\mu_a 
&\leq \frac{1}{\mu(B)}\frac{1}{\nu_a(I)}
\int_B\int_I \abs{u(x,y)-u_B(y)} + \abs{u_B(y)-u_D} d\mu(x) d\nu_a(y).
\end{align*}
For the first part, 
\begin{align*}
\frac{1}{\nu_a(I)}\int_I
\frac{1}{\mu(B)}\int_B \abs{u(x,y)-u_B(y)} d\mu(x)d\nu_a(y) &\leq C \,\text{rad}\,(B)\frac{1}{\nu_a(I)}\int_I
\brak{\frac{1}{\mu(\lambda B)}\int_{\lambda B} d\Gamma(u,u)}^{1/2}d\nu_a(y)\\
&\leq C \,\text{rad}\,(B)
\frac{1}{\nu_a(I)}
\brak{\frac{1}{\mu(\lambda B)}\int_I\int_{\lambda B} d\Gamma(u,u)  d\nu_a(y)}^{1/2} \sqrt{\nu_a(I)}\\
&\leq \tilde{C} \,\text{rad}\,(B)
\brak{\frac{1}{\mu_a(\lambda D)}\int_{\lambda D} d\Gamma(u,u) d\nu_a(y)}^{1/2} 
\end{align*}
where the second inequality follows from the H\"{o}lder's inequality and the last follows from the doubling property.
Similarly, 
\begin{align*}
\frac{1}{\mu(B)}\frac{1}{\nu_a(I)}
\int_B\int_I  \abs{u_B(y)-u_D} d\mu(x) d\nu_a(y) \leq 
\tilde{C} \,\text{rad}\,(B)
\brak{\frac{1}{\mu_a(\lambda D)}\int_{\lambda D} \abs{\frac{\partial u(x,y)}{\partial y}}^2 d\nu_a(y)d\mu(x)}^{1/2}
\end{align*}
Notice that 
\begin{align}\label{Gammaa}
d\Gamma_a(u,u) = d\Gamma(u,u)d\nu_a(y) + \abs{\frac{\partial u(x,y)}{\partial y}}^2 d\nu_a(y)d\mu(x),
\end{align}
then we achieve the Poincar\'{e}'s inequality for the region $D$. Apply the doubling property with $\tilde{\lambda}$ large enough to let $\tilde{\lambda}B_a$ covers $D$, we have 
\begin{equation*}
\frac{1}{\mu(B_a)}\int_{B_a} \abs{u-u_D} d\mu_a \leq C \,\text{rad}\,(B_a) \brak{\frac{1}{\mu(\lambda B_a)}\int_{\lambda B_a} d\Gamma_a(u,u)}^{1/2}
\end{equation*}
which is exactly statement \eqref{poincarec} with $c = u_D$. Hence \eqref{poincare} follows from the claim at the beginning.

\end{proof}

\subsection{Harnack inequalities for $(-L)^s$}

From now on, we assume that $(X,\mathcal{E})$ is doubling and satisfies the 2-Poincar\'e inequality.  As in \cite{Caffarelli_2007}, let us first consider the following extension lemma.
\begin{lem}\label{evenext}
We consider  a solution $f\ge 0$, $f \in D((-L)^s)$  to $(-L)^s f = 0$ in $B(x,R)$. The fact that $(-L)^s f = 0$ in $X$ implies that the function $\tilde U(x,y) = U(x,|y|)$ actually is harmonic in  $B(x,R) \times \R$.
\end{lem}

\begin{proof}
By a density argument, it is enough to show that for all continuous $h \in \mathcal{F}_a$ whose compact support is included in $B(x,R) \times \R$,
$$
\E_a(\tilde{U}, h) = \int_{X_a} (-L\UT \cdot h + \UT_y\cdot h_y)d\mu d\nu_a = 0.
$$
Fix $h$, and let  $B=B(x,R) \times (-M,M)$ that contains $\text{supp} \ h$. For some small $\epsilon > 0$, We can write 
\begin{align*}
    \int_{X_a} (L\UT \cdot h - \UT_y\cdot h_y)\ d\mu d\nu_a &=
\int_{B \cap \crl{\abs{y} \geq \epsilon}} (L\UT \cdot h + \UT_y\cdot h_y)\ d\mu d\nu_a + 
\int_{B \cap \crl{\abs{y} < \epsilon}} (L\UT \cdot h - \UT_y\cdot h_y)\ d\mu d\nu_a \\
&:= I + J.
\end{align*}

For part $I$, since the region is away from the hyperplane $\crl{y=0}$, we can simply do integration by parts and apply the fact that $\UT$ is the weak solution to the equation (\ref{extL}).

The integration region of  part $I$ is away from the hyperplane $\crl{y=0}$, and the part $J \to 0$ as $\epsilon \to 0$.

\begin{align*}
    I = \int_{B \cap \crl{\abs{y} \geq \epsilon}} (L\UT \cdot h - \UT_y\cdot h_y)\ d\mu d\nu_a = \int_{B \cap \crl{\abs{y} \geq \epsilon}}
    L\UT \cdot h
    \ d\mu d\nu_a - 
    \int_{B \cap \crl{\abs{y} \geq \epsilon}}
    \UT_y\cdot h_y
    \ d\mu d\nu_a.
\end{align*}
Integration by parts on the second term yields,
\begin{align*}
    \int_{B \cap \crl{\abs{y} \geq \epsilon}}
    \UT_y\cdot h_y
    \ d\mu d\nu_a 
    &= 
    \int_{B \cap \crl{\abs{y} \geq \epsilon}}
    \UT_y\cdot h_y \cdot \abs{y}^a
    \ d\mu dy 
    \\
    &=
    \int_{B \cap \crl{\abs{y} \geq \epsilon}}
    \partial_y\brak{\abs{y}^a\UT_y } h
    \ d\mu d\nu_a
    -
    \int_{\partial\brak{{B \cap \crl{\abs{y} \geq \epsilon}}}}
    \abs{y}^a\UT_y  h
    \ dS\\
    &=
    \int_{B \cap \crl{\abs{y} \geq \epsilon}}
    \B \UT h
    \ d\mu d\nu_a
    -
    2\int_{{B \cap \crl{\abs{y} = \epsilon}}}
    \abs{y}^a\UT_y  h
    \ dS\\
\end{align*}
The last equality follows from $h|_{\partial{B}} = 0$. Since $\UT$ is the even extension of the solution to the equation (\ref{extL}), we have 

\begin{align*}
    I = \int_{B \cap \crl{\abs{y} \geq \epsilon}} (L\UT - \B\UT)h_y\ d\mu d\nu_a  +
    2\int_{{B \cap \crl{\abs{y} = \epsilon}}}
    \abs{y}^a\UT_y  h
    \ dS = 2\int_{{B \cap \crl{\abs{y} = \epsilon}}}
    \abs{y}^a\UT_y  h
    \ dS
\end{align*}
And $\lim_{\epsilon\to 0}I = 0$ follows from the assumption that $(-L)^s f = 0$ and Lemma \ref{dn}. 

By the proof of Lemma \ref{PoiF}, $L\UT$ and $\UT_y$ are locally integrable, hence $\lim_{\epsilon\to 0}J = 0$. Since the statement hold for arbitrary $\epsilon >0$, we have the desired result.
\end{proof}
%
%
%

We are ready to prove the Harnack's inequality for solutions of $(-L)^sf=0$.

\begin{thm}(Harnack's inequality)\label{Harnack_frac}
Assume that $(X,d,\mu,\mathcal{E})$ satisfies the doubling condition and the 2-Poincar\'e inequality. There exist a constant $C>0$ and $\delta \in (0,1)$ such that for any ball $B(x,R)\subset X$ and any non-negative function $f \in D((-L)^s)$ satisfies $(-L)^sf=0$ on  $B(x,R)$, 
\[
 \sup_{z \in B(x,\delta R)} f(z) \le C \inf_{z \in B(x,\delta R)} f(z).
\]
%
\end{thm}

\begin{proof}
We first note that the extended Dirichlet space $X_a$ satisfies the conclusion of theorem \ref{Harnack elliptic}.
Apply the extension lemma \ref{evenext}, we have a function  $\tilde{U}$ in the extended space $X_a$ which is harmonic in $B(x,R) \times \R$. 
Let $B$ be a ball in $X$, there exists a ball $\tilde{B}$ in $X_a$, such that $$\tilde{B} \cap (X\times \{0\}) = B(x,R)\times \{ 0 \}.$$
Since $\tilde{U}(\cdot, 0) = f$, we have 
$$
\sup_{z\in B(x,\delta R) } f(z) = \sup_{z \in B(x,\delta R)} \tilde{U}(z,0) \leq \sup_{\delta \tilde{B}}\, \tilde{U}  $$
and
$$
\inf_{ \delta \tilde{B}} \,  \tilde{U} \leq \inf_{z \in B(x,\delta R) } \tilde{U}(x, 0) = \inf_{z \in B(x,\delta R) }  f(z)$$
Also, $\tilde{U}$ is non-negative because of $f$. By the Harnack's inequality of $\tilde{U}$, we get the desired result.
\end{proof}
%


\subsection{Boundary Harnack principles}\label{BdyHarnack}
%
%
We will establish the boundary Harnack principles for $(-L)^s$ by applying the result in \cite{lierl2014}. We first address some important prerequisite definition and assumptions about non-symmetric Dirichlet forms. Note that we will derive an simpler version for symmetric case.

\begin{definition}
	Let $(\E,\F)$ be a local, regular Dirichlet form on $L^2(X, \mu)$. Then  $\E^{sym}(u,v) = (1/2)(\E(u,v) + \E(v,u))$ is its symmetric part and $\E^{skew}(u,v) = (1/2)(\E(u,v) - \E(v,u))$ is its skew-symmetric part. 
\end{definition}

\begin{prop}
The symmetric part $\E^{sym}$ of a local, regular Dirichlet form can be writen uniquely as 
$$
\E^{sym}(f,g) = \E^s(f,g) + \int f g d \kappa,\text{ for all } f,g\in \F,
$$
where $\E^s$ is strongly local and $\kappa$ is a positive Radon measure. The second term is also called the killing part. 	
\end{prop}

With respect to $\mathcal{E}$ we can define the following \emph{intrinsic metric} $d_{\mathcal{E}}$ on $X$ by
\begin{equation}\label{eq:intrinsicmetric}
d_{\mathcal{E}}(x,y)=\sup\{u(x)-u(y)\, :\, u\in\mathcal{F}\cap C_0(X)\text{ and } d\Gamma(u,u)\le d\mu\}.
\end{equation}
Here the condition $d\Gamma(u,u)\le d\mu$ means that $\Gamma(u,u)$ is absolutely continuous with 
respect to $\mu$ with Radon-Nikodym derivative bounded by $1$. 

The term ``intrinsic metric'' is potentially misleading because in
 general there is no reason why $d_{\mathcal{E}}$ is a metric on $X$ (it could be infinite for a given pair of points $x,y$
or zero for some distinct pair of points), however in this paper we will work in a standard setting in which it is a metric.
The following definition is from \cite{LenzStollmannVeselic} and references therein, which is based on the classical papers of K. T. Sturm
\cite{St-II,St-III,St-I}).

For the introduction of the following Theorem, we fix a symmetric strongly local regular Dirichlet form $(\hat \E, \F)$ on $L^2(X, \mu)$ with energy measure $\hat \Gamma$. Let $Y$ be an open subset of $X$. Let $(\E, D(\E))$ be another (possibly non-symmetric) local bilinear form on $L^2(X, \mu)$. We need the following two assumptions.

\begin{assump}
	\begin{enumerate}
		\item $(\E, D(\E))$ is a local, regular Dirichlet form. Its domain $D(\E)$ is the same as the domain of the form $(\hat \E,\F)$, that is, $D(\E) = \F$. Let $C_0$ be the constant in the sector condition for $(\E, \F)$, i.e.
		$$ |\E^{skew}(u, v)|  = |(1/2)(\E(u,v) - \E(v,u))| \leq C_0 (\E_1(u,u))^{1/2}(\E_1(v,v))^{1/2},
		$$
		for all $u, v, \in \F$, where $\E_1(f,g) = \E(f,g) + \int_Xfg d\mu.$
		\item There is a constant $C_1 \in (0, \infty)$ so that for all $f,g\in \F_{loc}(Y)$ with $f^2\in\F_{c}(Y)$, 
		$$ C_1^{-1} \int f^2 d\hat \Gamma (g,g) \leq \int f^2 d\Gamma(g,g) \leq C_1 \int f^2 d\hat \Gamma(g,g).$$
		
		\item There are constants $C_2, C_3\in[0, \infty)$ so that for all $f \in \F_{loc}(Y)$ with $f^2 \in \F_c(Y)$,
		$$
		\int f^2 d\kappa  \leq 2\brak{\int f^2 d\mu}^{1/2}\brak{C_2 \int d\hat \Gamma(f,f) + C_3\int f^2 d\mu}^{1/2}
		$$
		
		\item There are constants $C_4, C_5\in[0, \infty)$ sso that for all $f \in \F_{loc}(Y)\cap L_{loc}^\infty(Y), g\in\F_c(Y)\cap L^\infty(Y),$		
		$$
		|\E^{skew}(f, fg^2) | \leq 2\brak{\int f^2d\hat \Gamma(g,g)}^{1/2}
		\brak{C_4\int g^2 d\hat \Gamma(f,f) + C_5 \int f^2g^2d\mu}^{1/2}
		$$
	\end{enumerate}
\end{assump}
\begin{assump}
	There are constants $C_6, C_7 \in [0, \infty)$ such that 
	\begin{align*}
			|\E^{skew}(f, f^{-1}g^2) | \leq &2\brak{\int d\hat \Gamma(g,g)}^{1/2}
	\brak{C_6 \int g^2 d\hat \Gamma(\log f, \log f)}^{1/2}\\
	&+
	2\brak{\int d\hat \Gamma(g,g) + \int g^2 d\hat \Gamma(\log f, \log f)}^{1/2}
	\brak{C_7\int g^2 d\mu}^{1/2}
	\end{align*}
\end{assump}

Secondly, we introduce the notion of (inner) uniformity. Let $\Omega \subset X$ be open and connected. Recall that the $inner\ metric$ on $\Omega$ is defined as 
$$
d_\Omega(x,y) = \inf \crl{ \text{length}(\gamma)\ |\ \gamma:[0,1]\rightarrow \Omega \text{ continuous, } \gamma(0) = x, \gamma(1) = y },
$$
and $\tilde \Omega$ is the completion of $\Omega$ with respect to $d_\Omega$. For an open set $B\subset \Omega$, let $\partial_{\tilde \Omega} B = \bar{B}^{d_\Omega}\backslash B$ be the boundary of $B$ with respect to its completion for the metric $d_\Omega$. If $x$ is a point in $\Omega$, denote by $\delta_\Omega(x) = d(x, X\backslash \Omega)$ the distance from $x$ to the boundary of $\Omega$. 

\begin{defi}
Let $\gamma:[\alpha, \beta] \rightarrow \Omega$ be a rectifiable curve in $\Omega$ and let $c\in (0,1), C\in(1,\infty)$. We call $\gamma $ a $(c,C)-uniform$ curve in $\Omega$ if 
$$
\delta_\Omega(\gamma(t)) \geq c\cdot \min\crl{d(\gamma(\alpha), \gamma(t)), d(\gamma(t), \gamma(\beta))}, \text{ for all } t\in[\alpha,\beta],
$$
and if 
$$\text{length}(\gamma) \leq C\cdot d(\gamma(\alpha), \gamma(\beta)).$$
The domain $\Omega$ is called $(c, C)-uniform$ if any two points in $\Omega$ can be joined by a $(c, C)-uniform$ curve in $\Omega$.
\end{defi}
In the following discussion, we suppose $\Omega$ is a $(c_u, C_u)$-inner uniform domain in$(X, d)$.
Then we are ready to introduce the Theorem 4.2 of \cite{lierl2014}.

\begin{thm}\label{BH:ref}
    Let $(X, \mu, \hat\E, \F)$ be a strongly local regular symmetric Dirichlet space and $Y$ be an open subset of $X$.
	Suppose the Volume doubling property (Definition \ref{VD}), Poincar\'{e} inequality (Definition \ref{PI}) hold, together with the following two properties.   
   	\begin{align}
		\text{\parbox{.85\textwidth}{The intrinsic distance $d$ is finite everywhere, continuous, and defines the original topology of $X$.}}\tag{A1}\\
		\text{\parbox{.85\textwidth}{For any ball $B(x, 2r) \subset Y$, $B(x,r)$ is relatively compact. }}\tag{A2-Y}
	\end{align} 
    Suppose $(\E, \F)$ is another Dirichlet form satisfying the Assumptions 1 and 2.     Let $\Omega \subset Y$ be a bounded inner uniform domain in $(X, d)$. There exists a constant $A_1\in (1, \infty)$ such that for any $\xi \in \partial_{\tilde \Omega}\Omega$ with $R_\xi >0$ and any
    $$
    0 < r < R \leq \inf\{R_{\xi'}: \xi'\in B_{\tilde\Omega}(\xi, 7R_\xi)\backslash \Omega\},
    $$
    and any two non-negative weak solutions $u,\ v$ of $L u = 0$, where $L$ is the generator of $\E$,  in $Y' = B_\Omega(\xi, 12 C_\Omega r)\backslash \Omega$ with weak Dirichlet boundary condition along $B_{\tilde \Omega}(\xi, 12 C_\Omega r)\backslash \Omega$,  we have
    $$
    \frac{u(x)}{u(x')}\leq A_1 \frac{v(x)}{v(x')},$$
    for all $x, x'\in B_\Omega(\xi, r)$. The constant $A_1$ depends only on the volume doubling constant, Poincar\'{e} constant, the constants $C_0-C_7$ which give control over the skew-symmetric part and the killing part of the Dirichlet form, the inner uniformity constants $c_u$, $C_u$ and an upper bound on $C_8R^2$.

\end{thm}

We wish to prove the Boundary Harnack principles over the extended space $(X_a, d\mu_a, \E_a, \F_a)$. We shall see the following property of the intrinsic metric is the only thing left to check.

Let $X_a = X \times \R$ and defined the natural product distance by 
\begin{align}\label{eq:da}
	d_a(z,w)^2 = d (z_x, w_x)^2+ |z_y - w_y|^2,
\end{align}
    where $z,w\in X_a$. For a point $z \in X$, we denote $z = (z_x, z_y)$ ,where $z_x \in X$ and $z_y \in \R$. The intrinsic distance for the Dirichlet space $(X_a, d\mu_a, \E_a, \F_a)$ is defined as 
\begin{align}\label{eq:dEa}
	d_{\mathcal{E}_a}(z,w)=\sup\{u(z)-u(w)\, :\, u\in\mathcal{F}_a\cap C_0(X_a)\text{ and } d\Gamma_a(u,u)\le d\mu_a\}.
\end{align}
Recall the definition for $d\Gamma_a$ in \eqref{Gammaa}.

\begin{lem}\label{lem:dEa}
   	The intrinsic metric $d_{\mathcal{E}_a}$ \eqref{eq:dEa}  is equivalent to the natural product metric $d_a$ \eqref{eq:da}.
\end{lem}

\begin{proof}
    First we prove $d_{\mathcal{E}_a}(z,w) \gtrsim d_a(z,w)$ for all $z,\ w \in X_a$. Note that $z = (z_x, z_y)$, $w = (w_x, w_y)$ where $z_x, w_x \in X$ and $z_y, w_y\in \R$. For some $0<\delta<1$, there exists $f \in \F^{loc}(X)\cap C(X)$ and $d\Gamma(f,f)\leq d\mu$, such that 
    $$
    f(z_x) - f(w_x) \geq \delta d(z_x, w_x).
    $$
    Let $F(z) = \frac{1}{2}(f(z_x) + z_y)$. It is clear that $F\in \F_a^{loc}(X_a) \cap C(X_a)$ and $d\Gamma_a(F,F)\leq d\nu_a d\mu$.  We have
    \begin{align*}
    d_{\E_a}(z,w) 
    &\geq \frac{1}{2}
    \brak{f(z_x) - f(w_x)} + \frac{1}{2}\brak{z_y - w_y}
    \geq 
    \frac{\delta}{2} d_\E(z_x,w_x)  + \frac{1}{2}\brak{z_y - w_y}\\
    &\geq 
    \frac{c \delta }{2} d(z_x,w_x)  + \frac{1}{2}\brak{z_y - w_y},
    \end{align*}
    where $c$ is the equivlent constant for $d$ and $d_\E$. And similarly for $\tilde F(z) = \frac{1}{2}(f(z_x) - z_y)$, we have     
    $
        d_{\E_a}(z,w) \geq
        \frac{c \delta }{2} d(z_x,w_x) + \frac{1}{2}\brak{w_y - z_y}.
    $
    Hence there exists a constant $C>0$ such that
    \begin{align*}
        d_{\E_a}(z,w) \geq \frac{c \delta }{2} d(z_x,w_x)  + \frac{1}{2}\abs{z_y - w_y} \geq C \sqrt{d(z_x, z_y)^2 + \abs{z_y, w_y}^2} = d_a(z, w).
    \end{align*}
    
    Then we prove the other direction. For a fixed $0<\delta<1$,  any $z, \ w \in X_a$, there exists $F\in \F_a^{loc}(X_a)\cap C(X_a)$ and $d\Gamma_a(F, F)\leq d\nu_a d\mu$, such that 
    $$ F(z) - F(w) \geq \delta d_{\E_a}(z, w).$$
    We wish to prove $d_a(z,w)\geq F(z) - F(w).$ Since $$F(z) - F(w) = F(z_x, z_y) - F(w_x, z_y) + F(w_x, z_y) - F(w_x, w_y).$$
    Let $f(x) = F(x, z_y)$, it is clear that $f\in\F^{loc}(X)\cap C(X)$ and $\frac{d\Gamma(f,f)}{d\mu}(x)\leq \frac{d\Gamma(F,F)}{d\mu}(x, z_y)\leq 1$. Hence 
    $$d_\E(z_x, w_x)\geq f(z_x) - f(w_x) = F(z_x, z_y) - F(w_x ,z_y).$$
    Let $g(y) =F(w_x, y)$, notice that $\abs{g'(y)} = \abs{F_y(w_x, y)}\leq 1.$ Then we have 
    $$ z_y - w_y \geq F(w_x, z_y) - F(w_x, w_y).$$
    Thus $d_a(z,w)\geq \tilde C d_\E(z,w)$ for some constant $\tilde C$.
\end{proof}

Now we are ready to prove the Boundary Harnack principle for the weak solutions of $(-L)^s$.

\begin{thm}(Boundary Harnack principles)
    Suppose $(X,d,\mu,\mathcal{E})$ is a symmetric, strongly local Dirichlet form  with generator $L$ satisfying the doubling condition and the 2-Poincar\'e inequality. Let $\Omega \subset X$ be a bounded inner uniform domain. There exists a constant $C>0$ such that for any $\xi \in \partial_{\tilde\Omega}\Omega$ with $R_\xi>0$ and any
    $$
    0 < r < R \leq \inf\{R_{\xi'}: \xi'\in B_{\tilde\Omega}(\xi, 7R_\xi)\backslash \Omega\},
    $$
    and any two non-negative weak solutions $u,\ v$ of $(-L)^s u = 0$ in $Y' = B_\Omega(\xi, 12 C_\Omega r)\backslash \Omega$, we have
    $$
    \frac{u(x)}{u(x')}\leq C \frac{v(x)}{v(x')},$$
    for all $x, x'\in B_\Omega(\xi, r)$. The constant $C$ depends only on the volume doubling constant, Poincar\'{e} constant and the inner uniformity constants $c_u, C_u$.
    \end{thm}

\begin{proof}
We apply the Theorem \ref{BH:ref}. First, we make $\E$ and $\hat \E$ in the Theorem \ref{BH:ref} coincide with $\E_a$, which is symmetric, strongly local and regular by Proposition \ref{prop:Ea}, hence the Assumption 1 and 2 are automatically satisfied. Secondly, since we assume the volume doubling property and the Poincar\'{e} inequality for the underlying form $(\E, \F)$, they can be extended to the form $(\E_a, \F_a)$ by Theorem \ref{ext:VD} and \ref{ext:PI}. We also make the open set $Y$ in Theorem \ref{BH:ref} to be the entire extended space $X_a$, then the condition (A1) and (A2-Y) are satisfied by the Lemma \ref{lem:dEa}. Now we can establish a boundary Harnack principle on the extended space $X_a$ for $\E_a$. By restricting the boundary value on the space $X$, which is similar to the steps in Theorem \ref{Harnack_frac}, we finish the proof. 
\end{proof}

\section*{Acknowledgments}

F.B. is partially funded by NSF grant  DMS-1901315, Q.L. would like to thank University of Connecticut for its hospitality during the preparation of this work. There is no data associated to the present article.

\bibliographystyle{plain}
\bibliography{Refs}

\end{document}